\documentclass[11pt,reqno]{amsart}
\usepackage{amsfonts, amsbsy, amsmath, amsthm, amssymb, latexsym, verbatim, enumerate}

\newtheorem{propo}{Proposition}[section]

\newtheorem{lemma}[propo]{Lemma}
\newtheorem{corol}[propo]{Corollary}

\newtheorem{theo}[propo]{Theorem}

\newcommand{\bl}{\begin{lemma}\label}
\newcommand{\el}{\end{lemma}}
\newcommand{\ld}{,\ldots ,}

\newcommand{\Id}{\mathop{\rm Id}\nolimits}

\newcommand{\CC}{\mathop{\mathbb C}\nolimits}

\newcommand{\al}{\alpha}
\newcommand{\be}{\beta}
\newcommand{\Ga}{\Gamma}

\newcommand{\lam}{\lambda }

\newcommand{\si}{\sigma }

\newcommand{\up}{^{-1}}

\newcommand{\bp}{\begin{proof}}
\newcommand{\enp}{\end{proof}}

\def\d12{{_{12}}}

\def\ei{{eigenvalue }}
\def\eis{{eigenvalues }}

\newcommand{\med}{\medskip}

\def\ii{{if and only if }}

\def\ir{{irreducible }}

\def\irr{{irreducible representation }}

\def\itf{{It follows that }}

\def\mult{{multiplicity }}

\def\rep{{representation }}
\def\reps{{representations }}

\def\Sym{{\rm Sym}}
\def\Alt{{\rm Alt}}
\def\Aut{{\rm Aut}}

\newcommand{\dne}{\hfill $\Box$   \\[-0.3cm]}

\newcommand{\gG}{\Gamma}
\newcommand{\g}{\gamma}

\newcommand{\gl}{\lambda}

\newcommand{\pf}{{\it Proof: \,}}

\newcommand{\Cay}{{\rm Cay}}

\newcommand{\x}{{\rm x}}

\newcommand{\cp}{\mathbb{C}}

\begin{document}

\title[On the second largest eigenvalue]{On the second largest eigenvalue of some \\Cayley graphs of the Symmetric Group}

\author[J. Siemons]{Johannes Siemons}
\address{School of Mathematics, University of East Anglia, University Plain, Norwich, NR47TJ, E-mail: J.Siemons@uea.ac.uk (Johannes Siemons)}

\author[A. Zalesski]{Alexandre Zalesski}
\address{ Department of Physics, Mathematics and Informatics, National Academy of Sciences of Belarus, 66 Prospekt  Nezavisimosti, Minsk, Belarus,
E-mail:   alexandre.zalesski@gmail.com (Alexandre Zalesski)}

\maketitle
\centerline{{\it Dedicated to the memory of Gordon James}}

\bigskip\medskip

{\it Abstract:}\, Let $S_n$ and $A_{n}$ denote the symmetric and alternating group on the set  $\{1\ld n\},$ respectively. In this paper we are interested in the second largest eigenvalue $\lambda_{2}(\gG)$ of the Cayley graph $\gG=\Cay(G,H)$ over $G=S_{n}$ or $A_{n}$  for certain connecting sets $H.$

Let $1<k\leq n$ and denote the set of all $k$-cycles in $S_{n}$ by $C(n,k).$ For $H=C(n,n)$ we prove that $\lambda_{2}(\gG)=(n-2)!$ (when $n$ is even) and  $\lambda_{2}(\gG)=2(n-3)!$  (when $n$ is odd). Further, for $H=C(n,n-1)$ we have  $\lambda_{2}(\gG)=3(n-3)(n-5)!$ (when $n$ is even) and  $\lambda_{2}(\gG)=2(n-2)(n-5) !$  (when $n$ is odd). The case $H=C(n,3)$ has been considered in \cite{HH}.

Let   $1\leq r<k<n$ and let $C(n,k;r) \subseteq C(n,k)$ be set of all $k$-cycles in $S_{n}$ which move all the points in the set $\{1,2\ld r\}.$ That is to say, $g=(i_{1},i_{2}\ld i_{k})(i_{k+1})\dots(i_{n})\in C(n,k;r)$ if and only if $\{1,2\ld r\}\subset \{i_{1},i_{2}\ld i_{k}\}.$
Our main result concerns  $\lambda_{2}(\gG)$, where $\gG=\Cay(G,H)$ with $H=C(n,k;r)$ with $1\leq r<k<n$ when $G=S_{n}$ if $k$ is even  and $G=A_{n}$ if $k$ is odd.  Here we observe that $$\lambda_{2}(\gG)\geq (k-2)! {n-r \choose k-r} \frac{1}{n-r} \big((k-1)(n-k) - \frac{(k-r-1)(k-r)}{n-r-1}\big).$$ We show that this bound is sharp in the special case $k=r+1$ , giving  $\lambda_{2}(\gG)=r!(n-r-1)$. The cases with $H=C(n,3;1)$ and $H=C(n,3;2)$  were considered earlier in \cite{HH}.


\section{Introduction}
Let $\gG$ be a finite undirected graph and let $A$ denote its adjacency matrix. The eigenvalues of this matrix, together with their multiplicities, are an important invariant of the graph. Since $A$ is symmetric all eigenvalues are real. For a regular graph of degree $d$ it is well-known that the largest eigenvalue of $A$ is $\gl_{1}(\gG)=d.$ The second largest eigenvalue, denoted by $\gl_{2}(\gG),$
plays an important role in many theoretical and practical applications of graph theory, from geometry to computer science.  The literature on the second largest eigenvalue is extensive, overviews and further references can be found in~\cite{An,BH, HLW, LZ,Lu,Sta}.
For regular graphs the spectral gap $\lambda_{1}(\Gamma)-\lambda_{2}(\Gamma)$ is known as the {\it algebraic connectivity} of $\gG.$

In this paper we are concerned with the second largest eigenvalue of certain Cayley graphs $\Cay(G,H)$ where $G$ is a symmetric or alternating group and where $H$ a connecting set in $G.$ (In particular,  we have $H=H^{-1}=\{h^{-1}\,|\,h\in H\},$ the identity element of $G$ does not belong to $H,$ and $H$ is not contained in any proper subgroup of $G.)$ For the definition of Cayley graphs see Section 2.4. References to recent work on the second largest eigenvalue of Cayley graphs over symmetric groups can be found in~\cite{HH}. Even for symmetric groups an explicit computation of it appears to be  one of the most challenging problem in the area.

To state our results let $S_n=\Sym(1\ld n)$ and $A_{n}$ be the symmetric and alternating group on the set  $\{1\ld n\},$ respectively.  For $1<k\leq n$ a $k$-cycle in $S_{n}$ is a permutation of the shape $g=(i_{1},i_{2}\ld i_{k})(i_{k+1})\dots(i_{n})$ with $|\{i_{1},i_{2}\ld i_{k}\}|=k.$ The set of all $k$-cycles in $S_{n}$ is denoted by $C(n,k).$

\begin{theo}\label{1A} Let $\gG=\Cay(G,H)$ with $H=C(n,n)$ for $n>4$ and $G=S_{n}$  when $n$ is even   and $G=A_{n}$ when $n$ is odd. Then $\lambda_{1}(\gG)=(n-1)!$. Furthermore, $\lambda_{2}(\gG)=(n-2)!$ when $n$ is even and $\lambda_{2}(\gG)=2(n-3)!$ when $n$ is odd.
\end{theo}

\begin{theo}\label{1B} Let $\gG=\Cay(G,H)$ with $H=C(n,n-1)$ for $n>4$ and $G=S_{n}$  when $n$ is odd    and $G=A_{n}$ when $n$ is even. Then $\lambda_{1}(\gG)=n(n-2)!$. Furthermore, $\lambda_{2}(\gG)=3(n-3)(n-5)!$ when $n$ is even and $\lambda_{2}(\gG)=2(n-2)(n-4)!$ when $n$ is odd.
\end{theo}

An essential feature of the Cayley graphs in these two theorems is  that the connecting set $H$ is invariant under conjugation by $G.$ In fact, $C(n,k)$ is a conjugacy class of $S_n$. If $k=n$ or $n-1$,
the above results are deduced from some analysis of the group characters in a relatively straightforward way. For arbitrary $k$ computing  the second largest eigenvalue for $\Cay(G,C(n,k))$ is an open problem.
For $k=3$ this is solved  in \cite[Theorem 3.4]{HH}, for $k=2$ see \cite[Lemma 3]{KY}.

 Next let $1\leq r<k<n$ and let $C(n,k;r) \subseteq C(n,k)$ be set of all $k$-cycles in $ \Sym(1\ld n)$ which move all the points in  the set $\{1,2\ld r\}.$ That is to say, $g=(i_{1},i_{2}\ld i_{k})(i_{k+1})\dots(i_{n})\in C(n,k;r)$ if and only if $\{1,2\ld r\}\subset \{i_{1},i_{2}\ld i_{k}\}.$ Clearly, $|C(n,k;r)|=(k-1)!{n-r\choose k-r}.$

 Our main results are Theorems \ref{th13-} and \ref{th13}. As $\lambda_{1}(\gG)=|H|$ in general (whenever $H$ is a connecting set), we focus on $\lambda_{2}(\gG)$. 

\begin{theo}\label{th13-}   Let $n>4$ be an integer and $1\leq r<k<n$.  Let $\gG=\Cay(G,H)$ with $H=C(n,k;r)$, where $G=S_{n}$ when $k$ is even and $G=A_{n}$ when $k$ is odd. Then 
$\mu_2=(k-2)! {n-r \choose k-r}\frac{1}{n-r} \big((k-1)(n-k) - \frac{(k-r-1)(k-r)}{n-r-1}\big)$ is an \ei of $\gG$
and hence $\lam_2\geq \mu_2$.\end{theo}

We conjecture that $\lam_2=\mu_2$ for any $r,k$ with $1\leq r<k<n$. We prove this conjecture  for $k=r+1$ for 
arbitrary $r$ with $1\leq r<n-1$:

\begin{theo}\label{th13}   Let $n>4$ be an integer and $2\leq r\leq n-2$.  Let $\gG=\Cay(G,H)$ with $H=C(n,r+1;r)$ where $G=S_{n}$ when $r$ is odd  and $G=A_{n}$ when $r$ is even. Then $\lambda_{1}(\gG)=r!(n-r)$ and $\lambda_{2}(\gG)=r!(n-r-1)$.
\end{theo}

In Theorems 1.3 and 1.4, in contrast to Theorems 1.1 and 1.3,  the connecting set $H$ is not invariant under $S_n.$ However, we make essential use of the fact that $H$ is invariant under  a subgroup isomorphic to $S_r  \times S_{n-r}$. Other cases where the bound is sharp include $\Gamma=\Cay(G,H)$ with $H=C(n,3;1)$ where $\lambda_{2}(\Gamma)=n^{2}-5n+5$ by \cite[Theorems 3.2]{HH}. The cases
$H=C(n,3;2)$ with $\lambda_{2}(\Gamma)=2n-6$ and $H=C(n,2;1)$ with $\lambda_{2}(\Gamma)=n-2$ were resolved earlier in   \cite[Theorem 2.3]{HH} and \cite[Section 8.1]{LZ} respectively.

\med
{\it Notation:} \,All groups considered here are finite. Further, all modules and \reps are over the field $\CC$ of complex numbers. By $\cp G$ we denote the group algebra of the group $G$ over $\CC$. If $H\subset G$ then $|H|$ is the number of elements of $H$ and
$H^+$ is the sum of all $h\in H,$ as an element of $\cp G$. A sentence such as\, `$L$ is a $G$-module' \,means that  $L$ is a $\cp G$-module.
If $L$ is a $\cp G$-module and $X$ is a subgroup of $G$, we write $L|_X$ for the restriction of $L$ to $X$. The  trivial  $G$-module and its character are denoted by $1_G$. We assume that the reader is familiar with general notions and elementary facts of finite group theory, including the \rep theory of $S_n$ and $A_n$.

\section{Preliminaries}

We collect several well-known facts and prerequisites for the paper. All graphs are finite, undirected and without multiple edges.

\subsection{\sc Graphs and Eigenvalues} Let $\gG$ be a graph on the vertex set $V;$
we put $m=|V|.$ For $u,v\in V$ we write $u\sim v$ if $u$ is adjacent to $v.$ For $v\in V$ the set $N(v)=\{u\in V\,:\,u\sim v\}$ denotes the set of {\it neighbours} of $v$ and $d(v)=|N(v)|$ is the {\it degree} of $v.$ The graph is {\it regular} of degree $d$ if $d(v)=d$ for all $v\in V.$ A permutation $\g$ of $V,$ $u\mapsto u^{\g},$ is an {\it automorphism} of $\Ga$ if $(N(v))^{\g}=N(v^{\g})$ for all $v\in V.$  The group of all automorphisms of $\gG$ is denoted by $\Aut(\gG).$

The {\it adjacency matrix} $A=(a_{u,v})$ of $\gG$ is given by $a_{u,v}=1$ if $u\sim v$ and $a_{u,v}=0$ otherwise. Then $A$  determines $\gG$. The  {\it eigenvalues} of $\gG$ are, by definition, the \eis of $A$.  These are real since $A$ is symmetric.

\begin{theo}\label{wk1}  {\rm \cite[Proposition 1.48]{kr}} Let $\gG$ be a regular graph of degree $d$ and let $A$ be its adjacency matrix.  Then \\[5pt]
$(1)$\, $d$ is an eigenvalue of $A.$ Furthermore,  $d$ has \mult $1$ if and only if $\gG$ is connected;  \\[5pt]
$(2)$\, $\gG$ is bipartite if and only if $-d$ is an eigenvalue of $A;$ \\[5pt]
$(3)$\, $d\geq |\lambda|$ for all eigenvalues $\lambda$ of $A.$ If $\gG$ is bipartite and if $\lambda$ is an eigenvalue of $A$ then $-\lambda$ is an eigenvalue of $A$ with the same multiplicity as $\lambda.$
\end{theo}

\subsection{\sc Equitable partitions}  Let $\Gamma$ be a graph with vertex set $V.$ Let $P$ be a partition   $V=V_1\cup \cdots\cup V_k$ of $V$. Then $P$ is called {\it equitable} if $ |N(v_i)\cap V_j|$ does not depend on the choice of $v_i\in V_i$ for all $1\leq i,\,j\leq k.$  If $P$ is equitable then the matrix $B=B(P)=( |V_i \cap N(v_j)|) $ with $1\leq i,j\leq k$ is called the {\it associated matrix} of $P.$ As usual, a second partition $P'$  is a {\it refinement of} $P$ if every class of $P$ is some union of classes of $P'.$

For any graph $\gG$ examples of equitable partitions arise from the orbits of a group of automorphisms of $\gG,$ see Lemma~\ref{AA23} below.  For this reason equitable partitions are also called {\it generalized orbits}. For incidence graphs of geometries  also the term {\it tactical decomposition} is used, see Dembowski's book~\cite{Dem}.

\med Let  $\cp V$ be the vector space with basis $V$ over $\cp.$ Then $A$ can be viewed as a linear transformation of $\cp V$, and then $Av=\sum_ {x\in N(v)}x$ for every $v\in V$.
Let $(\cdot,\cdot)$ be the standard bilinear form on $\cp V$ defined by $(u,v)=\delta_{u,v}$ (the Kronecker delta) for $u,v\in V$.  Let $V=V_{1}\cup \dots \cup V_{k}$ be a partition $P$ of  $ V$.
Set $w_i=\sum _{u\in V_i}u$ for $i=1\ld k$. Then $(w_i,Av)=|N(v)\cap V_i|$ (as $(w_i,u)=1$ for $u\in V$
\ii $u\in V_i$). Let $\cp P$ be the vector space with basis $w_1\ld w_k$. In this notation we have:

\bl{AA22}   Let $\Gamma$ be a graph with vertex set V and adjacency matrix $A$, and let $P$ be a partition $V=V_{1}\cup \dots \cup V_{k}$ of $V.$ Then $P$ is equitable if and only if $ \cp P$ is $A$-invariant (that is, $Ax\in \cp P$ for all $\x\in \cp P).$ Furthermore,  if $P$ is equitable then $B(P)$ is the matrix of the restriction of $A$ to
 $\cp P$ for the basis $w_{1},\dots, w_{k}.$ In particular, all eigenvalues of $B$ are \eis of $\Ga.$ If $\Ga$ is regular of degree $d$ then $d$ is an eigenvalue of $B.$

 In addition, if $P'$ is an equitable refinement of $P$ then the \eis of $B(P)$ are \eis of $B(P')$.\el

\pf For $i=1\ld k$ we have
\begin{equation}\label{a1n}Aw_i=\sum _{v\in V}(Aw_i,v)v=\sum _{v\in V}(w_i,Av)v=\sum _{v\in V}(w_i,\sum_{u\in N(v)}u)v=\sum _{v\in V}|N(v)\cap V_i|v.\end{equation}
To be in $\cp P$ the right hand side must be constant on every $V_j$, that is, $|N(v)\cap V_i|=|N(v')\cap V_i|$
for every $j\in \{1\ld k\}$ and every $v,v'\in V_j$. The converse and the second assertion of the lemma follows from (\ref{a1n}) as well. Finally, if $\Gamma$ is regular of degree $d$ then $A(w_1+\cdots + w_k)=A\sum_{v\in V}v=d\sum_{v\in V}v=d(w_1+\cdots + w_k)$, so $d$ is an \ei of $A$ on $\cp P$.

Let $V=V'_{1}\cup \dots \cup V'_{l}$ be a refinement $P'$ of $P$ and let as above $w_1'\ld w_l'$ be a basis of $\cp P'.$ Then every $w_j$ ($j=1\ld k$) is a linear combination of $w_1'\ld w_{l}'$, so
$\cp P\subset \cp P'$, whence the additional statement.
\dne

\bl{AA23}  Let $\Gamma$ be a graph and $G$ a group of automorphisms of $\Gamma.$ Let $V_1,\,...,\,V_k$ be the orbits of $G$ on the vertex set $V$ of $\Gamma.$  Then $V=V_1\cup \cdots\cup V_k$ is equitable.
\el

\bp For $i, j$ with $1\leq i, j\leq k$ let $v,\, v'\in V_{i}.$ Then there exists $g\in G$ so that $g(v)=v'$ and therefore $|N(v)\cap V_{j}|=|g(N(v)\cap V_{j})|=|N(g(v))\cap g(V_{j}))|=|N(v')\cap V_{j}|.$ \enp

\subsection{\sc Eigenvalue Inequalities} We denote the distinct eigenvalues of a graph    $\Gamma$ by  $$\gl_{1}>\gl_{2}>\dots>\gl_{m'}$$ for some $m'\leq m.$ If $\Ga$ is regular of degree $d$ then $\gl_{1}=d$ by Theorem~\ref{wk1}. The second largest eigenvalue, $\gl_{2},$ plays a distinguished role in graph theory as it provides bounds for the isoperimetric constant of $\Ga,$ see also graph expanders, the Kahzdan constant and Cheeger inequality in \cite{kr}. From Theorem~\ref{wk1} and Lemma~\ref{AA22} we immediately have

\bl{AA231} Let $P$ be an equitable partition of the regular graph $\Gamma$ and let $\beta_{2}$ be the second largest eigenvalue of $B(P).$ Then $\gl_{2}\geq \beta_{2}.$ \el

Upper and lower bounds on eigenvalues of a symmetric matrix over the reals can be obtained from a theorem of Hermann Weyl~\cite{Weyl}, see also \cite[Theorem 2.8.1]{BH}.

\begin{theo}[Weyl Inequality]\label{AA232} Let $C=A+B$ be symmetric $(m\times m)$-matrices, and let $\gamma_1\geq \cdots \geq \gamma_m,$
$\al_1\geq \cdots \geq \al_m,$ $\be_1\geq \cdots \geq \be_m$ be the \eis of $C,B,A,$ respectively. Then for $1\leq i,j\leq m$ we have \begin{equation}\label{eH1}
\gamma_{i+j-1}\leq \al_i+\be_j\  {\rm whenever}\,\, i+j-1\leq m
\end{equation} and \begin{equation}\label{eH1b}
\gamma_{i+j-m}\geq \al_i+\be_j\  {\rm whenever} \,\,  i+j-1\geq m.
\end{equation}
\end{theo}

We note in particular, for $i,\,j\in\{1,2\}$ we have
\begin{equation}\label{eH2} \gamma_{2}\leq \al_2+\be_1\,\,\text{and}\,\,\,\gamma_{2}\leq \al_1+\be_2.
\end{equation} We refer to (\ref{eH2}) as the Weyl inequality.

\med
For adjacency matrices we have several applications. First let $\g_{1}\geq \g_{2}\geq\g_{3}\geq \dots \geq \g_{m}$ be {\it all} eigenvalues of $\gG,$ with repetitions.  If $\gG$ is regular of degree $d$ and connected then we  have $d=\g_{1}=\gl_{1}>\gl_{2}=\g_{2}$ by Theorem~\ref{wk1}, for other eigenvalues repetitions may appear.

Let $v$ be a vertex of $\gG$ and let $\Gamma'$ be obtained from $\Gamma$ by deleting all edges incident with $v.$ Denote the adjacency matrix of $\Gamma'$ by $A'$ and let $\g_{1}'\geq \g_{2}'\geq \dots \geq \g_{m}'$ be {\it all} eigenvalues of $\gG',$ with repetitions. Next let $\Sigma$ be the `star' at $v$ with $d$ rays, where $d$ is the degree of $v,$ and with $m-d-1$ isolated vertices. Denote its adjacency matrix by $S.$ The eigenvalues of the star are $\sigma_{1}=\sqrt{d}>\sigma_{2}=0=\dots=\sigma_{m-1}>\sigma_{m}=-\sqrt{d},$ by a simple computation. Then $A=A'+S$ and from Theorem~\ref{AA232} we obtain the following

\begin{corol}\label{AA233} Let $v$ be a vertex of the connected  graph $\gG$ and let $\Gamma'$ be obtained by deleting all edges incident with $v.$ Let $\gl_{1}>\gl_{2}$ and $\gl_{1}'>\gl_{2}'$ be the largest and second largest eigenvalues of $\gG$ and $\gG',$ respectively.  Then \begin{equation}\label{Inta}
\gl_{1}\geq \gl'_{1}\geq \gl_{2}\geq \gl'_{2}. \end{equation}
Furthermore, if $d$ is the degree of $v$ then \begin{equation}\label{Interb}
|\gl_{i}-\gl_{i}'|\leq \sqrt{d}
\end{equation} for $1\leq i\leq 2.$
\end{corol}

Accounting for the isolated vertex $v,$ the spectrum of the graph $\Gamma[V\setminus\{v\}]$ induced on $V\setminus\{v\}$ is obtained from the spectrum of $\gG'$ by removing one eigenvalue $\gl'=0.$ The first part of the corollary is therefore an instance of the well-known interlacing theorem, see~\cite[Theorem 2.5.1]{BH}. However, the second part of the corollary appears to be new.

\subsection{\sc Cayley Graphs} Let $G$ be a group with identity element $1$ or $1_G$ and let $H$ be a subset of $G.$ Then $H$ is a  {\it connecting set} if (i) $1\not\in H,$\, (ii)  $H=H^{-1}(=\{h^{-1}\,|\,h\in H\})$ and\, (iii) $H$ is not contained in any proper subgroup of $G.$ For such a set $H$ the {\it Cayley graph} $\Cay(G,H)$ is the graph with vertex set $V=G$ so that two vertices $u,\,v\in G$ are adjacent if and only if $uv^{-1}\in H.$ Thus $N(v)=Hv$ and so $\Cay(G,H)$ is regular of degree $|H|$. The  adjacency matrix of $\Cay(G,H)$ is denoted by $A=A(G,H).$

More generally, one can construct a similar graph $\Gamma(G,H)$ if (i),(ii) holds and (iii) fails.
Then $\Gamma(G,H)$ is still a regular graph of degree $|H|$. Let $G_1$ be the least subgroup of $G$ containing $H$. Then $H$ is a connecting set  for $G_1$, and $\Gamma(G,H)$ is a disconnected graph isomorphic to the union of $|G:G_1|$ copies of $\Cay(G_1,H). $ In particular, the \eis of $\Gamma(G,H)$ and $\Cay(G_1,H)$ are the same (disregarding the multiplicities).
In fact, $A(G,H)$ is similar to a block diagonal matrix in which all blocks are  $A(G_{1},H).$ So we have:

\bl{an5}  Let $H$ be a subset of the group $G$  such that $1\notin H$ and $H=H\up$. Let $G_{1}$ be the smallest subgroup of $G$ containing $H.$ Then $\gG(G,H)$ and $\Cay(G_{1},H)$ have the same distinct eigenvalues. \el

Thus, such a generalization is not essential. However, for the study of \eis of $\Cay(G_1,H)$ by  means of \rep theory one sometimes prefers to deal with the adjacency matrix of $\Gamma(G,H)$ if the \rep theory of $G$ is simpler than that of $G_1$. In Section 6 we observe this for $G=S_n$ and $G_1=A_n$.

 There is no harm in identifying $g\in G$ with its matrix in the regular \rep of $G,$ of size $|G|\times |G|.$ A key fact for adjacency matrices of  Cayley graphs then is  that $\sum_{h\in H}\,h$, the sum over $h\in H$ of
the matrices just defined, is the adjacency matrix of $\Cay(G,H).$ Thus $A(G,H)=H^{+}$ under this identification. (Recall that we define $H^+$ as an element of the group algebra of $G$ over $\CC$.)
This is well known and explained, for instance, in \cite[p. 384]{SZ}. In other words,   $A$ is the image of $H^+$ in the regular \rep of the group algebra. Note that $A$ is a linear transformation
of the vector space $\cp V=\cp G$, the $\CC$-span of $G$.

\bl{AA24} Let $H$ be a connecting set in the group $G$ and let $F$ be the the normalizer of $H$ in $G.$ Then $F\times G$ acts as a group of automorphisms on $\gG=\Cay(G,H)$ by  $(f,g)\!:\,v\mapsto f^{-1}vg$ for all  $(f,g) \in F\times G$ and $v\in G.$ In particular, $\Aut(\gG)$ is vertex transitive. \el

\bp Let $u,v\in G.$ Then $uv^{-1}\in H$ if and only if $(f^{-1}ug)(g^{-1}v^{-1}f)\in f^{-1}Hf=H.$ Evidently $\{1_{F}\}\times G$ is transitive on $G.$  \enp

\section{\sc Some general results}

In this section we discuss representation theoretical aspects of eigenvalue problems for Cayley graphs.  Let $G$ be a finite group, $H\subset G$ a connecting subset, and put $H^+=\sum_{h\in H}$ as an element in the group algebra $\cp G$
of $G$ over $\cp.$ If $\rho$ is the regular \rep of $G$ then $A(G,H)=\rho(H^+).$ If we identify $g$ with $\rho(g)$ then $A(G,H)=H^+$ as discussed in Section 2.4.
 Thus, if $\phi$ is any representation of $G$ then $\phi(H^+)$ is meaningful, and the \eis of $\phi(H^+)$ are \eis of $\Gamma$. We think that this convention makes the exposition more transparent.

By standard results, $\rho$ is a direct sum of all \ir \reps $\phi$ of $G,$ each occurring with \mult $\dim \phi$. \itf  the set of all \eis of $Cay(G,H)$ (disregarding the multiplicities) is the union, over the \ir \reps $\phi$ of $G,$ of  the \eis  of $\phi(H^+)$.  See \cite[Proposition 7.1]{kr} or elsewhere.

If $\phi_0$ is the trivial \rep of $G$ then, obviously, $\phi_0(H^+)=|H|$. Conversely, if $|H|$
is an \ei of $\phi(H^+),$ where $\phi$ is an \irr of  $G$, then $\phi=\phi_0$ by Theorem~\ref{wk1} (as $\phi_0$ occurs in $\rho$ with \mult 1; here we use that $H$ is a connecting set).

\bl{sk1} Let $G_1$ be a subgroup of $G$ and let M be an \ir G-submodule
with character $\chi$. Let $H$ be a subset of $G.$\\[7pt]
$(1)$ If H is  G-stable (that is, $gHg\up=H$ for every $g\in G$) then the restriction  $H^+$ to M acts  as a scalar matrix
$\mu\cdot \Id$. If $H=h^G$ is a conjugacy class then $\mu=|H|\cdot \frac{\chi(h)}{\chi(1)}=|G:C_G(h)|\cdot \frac{\chi(h)}{\chi(1)}$. \\[7pt]
$(2)$ Let $L$ be an \ir $G_1$-module occurring in $M|_{G_1}$ exactly once.\\[5pt]
$(i)$ Suppose that H is $G_1$-stable. Then $H^+$ acts scalarly on L. \\[5pt]
$(ii)$ Suppose that $H=H_1\cup H_2$, where
$H_1,H_2$ are $G_1$-stable. Then  $H^+,H^+_1$ and $ H_2^+$ act scalarly on L, and $\nu=\nu_1+\nu_2$, where  $\nu,\nu_1,\nu_2$ are the \eis of  $H^+,H^+_1$ and $ H_2^+$ on L. This remains true if $H=H_1\cup \cdots \cup H_k$, where $H_1\ld H_k$ are $G_1$-normal. (That is, each $H^+_i$ acts scalarly on L, and $\nu=\nu_1+...+\nu_k$, where $\nu_i$ is the \ei of $H^+_i$ on L. \\[5pt]
$(iii)$ Suppose that $H_1=H\cap  G_1$ and that $H_1$ is a connecting set in  $G_1$. Then $\nu_1$ is the \ei of the Cayley graph $\Cay(G_1,H_1)$ occurring on $L$. In particular, if $H_1$ is a single conjugacy class of $G_1$ and  $\chi_L$ is the character of $L$ then $\nu_1=|H_1|\cdot \frac{\chi_L(h_1)}{\chi_L(1)}$, where $h_1\in H_1$.
 \el

\bp (1) is well known. (2) (i) Clearly,    $H^+$
commutes with every $g\in G_1$. Therefore, $H^+L$ is a $G_1$-module. By assumption,  $H^+L=L$. Then, by Schur's lemma,  $ H^+$ acts scalarly on $L$. \\[5pt]
 (ii) By (2), $ H_i^+$  stabilizes   $L$ and acts on it as $\nu_i\cdot \Id$ for $i=1,2$.   As $ H^+=H^+_1+ H^+_2$, the claim follows. \\[5pt]
(iii) Let $V_1=G_1$ be the vertices of the Cayley graph  $\Cay(G_1,H_1)$. As  $L$ is an \ir
$G_1$-module, $L$ is an \ir constituent of the regular $G_1$-module $\cp V_1=\cp G_1$, and hence the \ei of $H_1^+$ on $L$ occurs as an \ei of $H_1^+$ on  $\cp G_1$. The second claim in (iii) follows from (1).
\enp

Remarks. (1) Unfortunately, $\nu_1$ and $\nu_2$ depend not only on $L$ but on both $M,L$. In other words, if $L$ is a constituent of $M'$, an \ir $G$-module not isomorphic to $M$, then $\nu_1(M,L)$ may not be equal to $\nu_1(M',L)$. However, if $H_1=H\cap G_1$ then $\nu_1 $ depends on $L$ but does not depend on $M$.

(2) If $G=S_n,\,G_1=S_{n-1}$ then every \ir $ G$-module $M$ is \mult free as $ G_1$-module (that is, every \ir constituent of $M|_{S_{n-1}}$ occurs exactly once).
This follows from the branching rule expressed in terms of Young diagrams (we cannot get the same diagram by removing distinct single boxes from a given Young diagram. This remains true for alternating groups.) In addition, if $C$ is a conjugacy class of $S_n$ then  $C\cap S_{n-1}$ is either empty or a single conjugacy class of $G_1$, as two permutations $x,y\in G_1$ are conjugate in $S_n$ \ii they are conjugate in $G_1$. In the following we view $\Sym(1\ld k)$ as the largest subgroup of $S_{n}$ which fixes all $i\not\in \{1,\dots,k\},$ etc.

\def\esp{eigenspace }

\bl{ne1}
  Let $1\leq r<n-1$ and $G=S_{n}.$ Let $M$ be the natural permutation $G$-module  and  $F=\Sym(1\ld r)\times \Sym(r+1\ld n)\subset G.$
Suppose that the subset $H\subset G$ is normalized by $F,$ that is $xHx\up=H$ for all $x\in F.$   If $r\geq 2$ then  $H^+$ has at most $4$ distinct \eis on $M.$ Two of these have  \mult $1$,  and two  others, if distinct, have multiplicity $r-1$ and $n-r-1$. If $r=1$ then  $H^+$ has at most $3$ distinct \eis on $M.$ Two of these have  \mult $1,$ the remaining \ei has multiplicity  $n-2.$
 \el
 \bp Set $F_1=\Sym(1\ld r)\times \Id$ and $F_2=\Id\times \Sym( r+1\ld n)$. As $H$ is $F$-invariant, $H^+$ commutes with $F$ in the group algebra $\cp S_n$ and hence in ${\rm End}\, M$. We have $M|_F=M_1\oplus M_2$, where $M_i$ is an $F$-module such that $M_i|_{F_i}$
 the natural permutation modules for $F_i$ and $M_i|_{F_j}$ is trivial for $i\neq j$, $i,j=1,2$.  Then $M_1=T_1\oplus L_1$, where $T_1$ is the trivial $F_1$-module and $L_1$  is irreducible of dimension $ r-1$. In particular, $L_1=0$ when $r=1.$
 Similarly, $M_2=T_2\oplus L_2$,  where $L_2$ is irreducible of dimension $n-r-1$. Clearly, $L_1,L_2$ are non-equivalent non-trivial  $F$-modules. So $M|_F=L_1\oplus L_2\oplus T$, where $T$ is a trivial $F$-module of dimension 2.  By Schur's lemma,  $H^+$ acts scalarly on $L_1,L_2$, and hence $H^+$ has at most four distinct \eis on $M$. Clearly, if $r=1$ then $L_0=0$ and so $H^+$ has at most three distinct eigenvalues.
  \enp

\section{Symmetric group: the largest conjugacy classes}

For  basic definitions concerning the characters of symmetric  and alternating groups see~\cite{J,JK}.

Let $G=S_n$ and let $1\leq k\leq n.$ We denote the set of all $k$-cycles in $G$ by $C(n,k).$ Since two permutations  are conjugate to each other in $G$ if and only if they have the same cycle type we have $|C(n,k)|=|G:C_{G}(g)|$ when $g$ is any $k$-cycle and $C_{G}(g)$ denotes the centralizer of $g$ in $G.$ In particular,
$|C(n,n)|=(n-1)!$ and $|C(n,n-1)|=n(n-2)!.$ By the same argument, the enumeration of permutations by cycle type shows that all  other conjugacy class sizes in $G$ are smaller if $n>4$. It is well-known that the smallest subgroup of $S_{n}$ containing $C(n,k)$ is $S_{n}$ if $k$ is even and  $A_{n}$ if $k$ is odd.

In this section we  compute the second largest eigenvalue  of $H^+$ when $H=C(n,n)$ or $H=C(n,n-1).$ For this we first quote a well-known result for the cycles of length $n$ in \reps of $S_n$ \cite[Lemma 21.4]{J}.

Below let $\chi=\chi_{\mu}$ denote an \ir character of $G$ labeled by the  Young diagram $\mu$.

\bl{jj1} Let $h\in S_n$ be an $n$-cycle.
 Then $\chi(h)=0$ if $\mu$ is not a hook. If $\mu=[n-m,1^m]$ is a hook with leg length m then $\chi(h)=(-1)^m$.\el

Note that for $n\geq 4$ we have  $\chi_{[n-2,1^2]}(1)=(n-1)(n-2)/2$ and $\chi_{[n-2,2]}(1)=n(n-3)/2$,
so $\chi_{[n-2,2]}(1)+1=\chi_{[n-2,1^2]}(1)$.

\bl{jn1}  Let $h\in S_n$ be an $n$-cycle, $n>4$, and let $H=h^G$ be the conjugacy class of $h.$\\[8pt]
$(1)$ Suppose that $n$ is odd. Then the maximum of the \ir character ratio $\chi(h)/\chi(1)$ with $\chi(1)\neq 1$ is attained at
$\chi_{[n-2,1^2]}$ and equals  $2/(n-1)(n-2)$. The corresponding \ei of $H^+$  equals  $
2|H|/(n-1)(n-2)=2(n-3)!$.\\[8pt]
$(2)$ Suppose that n is even. Then the maximum of the \ir character ratio $\chi(h)/\chi(1)$ with $\chi(1)\neq 1$ is attained at
$\chi_{[2,1^{n-2}]}$ and equals $1/(n-1)$. The corresponding  \ei of $H^+$  equals $
|H|/(n-1)=(n-2)!$
\el

\bp By Lemma 3.1(1), it suffices to prove the claim on the ratio, and, by Lemma \ref{jj1}, only for the hook characters $\chi_{[n-m,1^m]}$ with $m$ even. In this case $\chi_{[n-m,1^m]}(h)=1$,
so we have to show that $\chi_{[n-m,1^m]}(1)$ is minimal for $m=2$ for $n$ odd and $n-2$ for $n$ even.

(1) Here $\chi_{[n-2,1^2]}(h)=1=\chi_{[3,1^{n-3}]}(h)$ and $\chi_{[n-2,1^2]}(1)= (n-1)(n-2)/2=\chi_{[3,1^{n-3}]}$. As $\chi_{[n-1,1]}(h)=\chi_{[2,1^{n-2}]}(h)=-1$,
it suffices to show that  $\chi_{[n-m,1^m]}(1)>(n-1)(n-2)/2$ for $2<m<n-3$.  For such $m$  we have

{\small
$$\chi_{[n-m,1^m]}(1)=\frac{n!}{n(n-m-1)!m!}
=\frac{(n-1)(n-2)\cdot...\cdot(n-m+1)}{2\cdot3\cdot...\cdot m}>\frac{(n-1)(n-2)}{2}.$$}

(2) In this case $\chi_{[n-1,1]}(h)=-1$, $\chi_{[2,1^{n-2}]}(h)=1$ as these characters differ by sign character multiple. In addition, $\chi_{[2,1^{n-2}]}(1)=n-1<(n-1)(n-2)/2<\chi_{[n-m,1^m]}(1)$ for $3<m<n-4$.
So the result follows. \enp
 \bl{jn2}  Let $h\in G=S_n$, $n>4$, be an $(n-1)$-cycle and let $H$ be the conjugacy class of $h.$ \\[8pt]
$(1)$ Suppose n is even.
Then the maximum of the \ir character ratio $\chi(h)/\chi(1)$ with $\chi(1)\neq 1$ is attained at
$\chi_{[n-3,2,1]}$ and equals $3/n(n-2)(n-4)$; the corresponding \ei of $H^+$ equals $
3|H|/n(n-2)(n-4)=3(n-3)(n-5)!$.\\[8pt]
$(2)$ Suppose n is odd. Then the maximum of the \ir character ratio $\chi(h)/\chi(1)$ with $\chi(1)\neq 1$ is attained at
$\chi_{[2,2,1^{n-4}]}$ and equals $2/n(n-3)$; the corresponding \ei of $H^+$  equals $
2|H|/n(n-3)=2(n-2)(n-4)!$.\el
\bp Note that 
$|H|=n!/(n-1)=n(n-2)!$. Let  $\chi=\chi_\mu$ be an \ir character of $G$  labeled by the Young diagram $\mu$. Suppose that $\chi(1)>1$. Let $h\in H$; we can assume that $h\in C_G(n)\cong S_{n-1}$.

$(i)$ We have $\chi(h)=0$ unless $\mu=[n-m,2,1^{m-2}]$ with $(2\leq m\leq n-2)$.
To show this let $\chi|_{S_{n-1}}=\nu_1+\cdots+\nu_k$, where $\nu_1\ld \nu_k$ are \ir character of $S_{n-1}$; then  $\chi(h)=\nu_1(h)+\cdots+\nu_k(h)$.
By Lemma \ref{jj1}, $\nu_i(h)=0$ unless $\nu_i$ corresponds to a hook. So $\chi(h)=0$ unless $\mu$ is obtained from a hook diagram by adding one box. Obviously, either $\mu$ itself is a hook or $\mu$ is obtained from a hook diagram by adding one box at the $(2,2)$-position, that is, $\mu=[n-m,2,1^{m-2}]$ for $m=2\ld n-2$.

 Suppose $\mu=[n-m,1^m]$ is a hook, $1\leq m\leq n-1$. Then $\chi|_{S_{n-1}}=\nu_1+\nu_2$, where the diagrams of $\nu_1,\nu_2$ are $[n-m-1,1^m]$ and $[n-m,1^{m-1}]$. The legs of these hooks are of distinct parity, so $\nu_1(h)+\nu_2(h)=0$ by Lemma \ref{jj1}.

$(ii)$ For $2\leq m\leq  n-2$ set $\chi_m=\chi_{\mu}$ for $\mu=[n-m,2,1^{m-2}]$.
Then $\chi_{\mu}(h)=(-1)^{m-1}$.

Indeed, if $2<m<n-1$ then $\chi_m|_{S_{n-1}}=\nu_1+\nu_2+\nu_3$, where $\nu_1,\nu_2,\nu_3$ are \ir characters of $S_{n-1}$ with Young  diagrams  $[n-m-1,2,1^{m-2}]$, $[n-m,2,1^{m-3}]$ and $[n-m,1^{m-1}]$, respectively.
So $\nu_1(g)=\nu_2(g)=0$ by Lemma \ref{jj1}, whereas  $\nu_3(g)=(-1)^{m-1}$.
So $\chi_m(h)=(-1)^{m-1}$. 
If $m=2$ then $\chi_2|_{S_{n-1}}=\nu_1+\nu_3$, whose Young  diagrams
 are $[n-3,2]$ and $[n-2,1]$. As above, $\nu_1(h)=0$, $\nu_3(h)=-1$ so $\chi_2(h)=-1$.
 Let $m=n-2$, so   $\mu=[2,2,1^{n-4}]$. Then  $\chi_{[2,2,1^{n-4}]}(h)=(-1)^n\chi_{[n-2,2]}(h) $ (because $h$ takes the value $-1$ at the sign non-trivial character \ii the permutation $h $ is even).
 So the claim follows.

Thus, $\chi_\mu(h)=1$ if $\mu=[n-m,2,1^{m-2}]$ with $m$ odd, and $\chi_\mu(h)<1$ otherwise (assuming $\chi_\mu(1)\neq 1$). Therefore, to complete the proof of the lemma it suffices to determine 
the minimum of $\chi_\mu(1):\,\mu=[n-m,2,1^{m-2}], \, m $\,\,odd.

$(iii)$ Let $2\leq m\leq (n-1)/2$. Then
\begin{eqnarray}
\chi_m(1)&=&\frac{n!}{(n-m-2)!(n-m)(n-1)m(m-2)!}\nonumber\\
< \,\,\,\chi_{m+1}(1)&=&\frac{n!}{(n-m-3)!(n-m-1)(n-1)(m+1)(m-1)!}
\end{eqnarray}
as $(n-m-1)(m+1)(m-1)<m(n-m)(n-m-2)$. (Indeed, $m\leq (n-1)/2$ implies $n-m\geq m+1$ and
$m(n-m)(n-m-2)\geq (n-m-2)m(m+1)=(n-m-1)m(m+1)-m(m+1)\geq(n-m-1)m(m+1)-(n-m-1)(m+1)=
(n-m-1)(m+1)(m-1)$.)

$(iv)$ By $(ii)$ and $(iii)$, the maximum of the ratio $\chi_m(h)/\chi_m(1)$ is attained for $m=n-2$ if $n$ is even, and for $n=3$ is $n$ is odd. This implies the result.\enp

We expect that the behaviour of $\chi_{[n-1,1]}(h)$ is not typical when $h$ is a $k$-cycle with $k=n,n-1.$ Namely, for $k<n-1$ we expect that the second largest \ei is attained for  $\chi=\chi_{[n-1,1]}$.

 By Lemma \ref{an5} we have:

 \bl{an2} For $n>4$ let $H$ be the set of all $n$-cycles in $A_n$ if $ n>3$ is odd and all $n-1$-cycles in $A_n$ if $n>5$ is even. Then $H$ is a connecting set in $V=A_n$ and the largest second \ei of $H^+$ on the Cayley graph $Cay(A_n,H)$ equals $2(n-3)!$ and $3(n-3)(n-5)!$ respectively. \el

 \bp[{\it Proof of Theorems {\rm \ref{1A} {\it and}~\ref{1B}}}] If $G=A_{n}$ the result follows from Lemma~\ref{an2}. If  $G=S_{n}$ then $\Cay(G,H)$ is bipartite, with partition $G=A_{n}\cup(S_{n}\setminus A_{n}),$ since $u\sim v$ implies that $uv^{-1}\in H\subset S_{n}\setminus A_{n}$ and so $u,\,v$ do not have the same parity.   Hence $|H|$ and $-|H|$ are eigenvalues of $\Cay(G,H)$ by Theorem~\ref{wk1}, corresponding to the trivial and the alternating characters of $S_{n},$ respectively. All other characters have degree $>1$ and so the result follows from Lemmas~\ref{jn1} and~\ref{jn2}. \enp

\section{\sc The natural $S_n$-module and equitable partitions}

Let $G=S_{n}$ and let $K$ be the stabilizer of some point in $\{1\ld n\}.$ Then the set $G/K$ of cosets of $K$ in $G$ and  $\{1\ld n\}$ are isomorphic as  $G$-sets. The associated permutation module over $\cp$ is the natural module for $G,$ denoted $M=1^{G}_{K}.$ This is probably the most important  $S_n$-module. Moreover, we expect that in 
many cases the second largest \ei of a Cayley graph over $S_{n}$ occurs on $M$ (although we have no means to justify this).

For $1\leq r<k\leq n$ let $C(n,k;r) \subseteq C(n,k)$ be the set of all $k$-cycles in $S_{n}$ which move all the points in the set $\{1,2\ld r\}.$ That is to say, $g=(i_{1},i_{2}\ld i_{k})(i_{k+1})\dots(i_{n})\in C(n,k;r)$ if and only if $\{1,2\ld r\}\subset \{i_{1},i_{2}\ld i_{k}\}.$ Clearly, $|C(n,k;r)|=(k-1)!{n-r\choose k-r}.$

\bl{n0}  Let $1\leq r<k\leq n$ and let X be the smallest subgroup of $S_n$ containing $C(n,k;r).$ Then $X=S_n$ if $k$ is even, and $X=A_n$ if $k$ is odd.\el

\bp This holds when $k=n$ since $C(n,n;r)=C(n,n)$ for all $r$ and so it suffices to show that $X\supseteq \Alt_{n}$ for  $1\leq r<k<n. $ Since for every  $i\in \{2\ld n\}$ there is some $h\in C(n,k;r)$ with $h(i)=1$ it follows that $X$ is transitive on $\{1,2\ld n\}.$ If $i<n$  we may assume additionally that $h(n)=n$ since $k<n.$ Hence the stabilizer of $n$ in $X$ is transitive on $\{1,2\ld n-1\}$ and so  $X$ is doubly transitive. The result follows if $k=2$ or $=3,$ since a double transitive group containing a $2$- or $3$-cycle contains $\Alt_{n}.$ For $k\geq 4$ consider the elements $h=(1,2\ld k-3,k-2,k-1,k)$ and $h'=(k-1,k,k-2,k-3\ld 2,1).$ Then $h'h=(1)(2)\dots (k-3)(k-2,k,k-1)\in X$ is a $3$-cycle and hence the result follows.\enp

\vspace{-0.2cm}
\begin{theo}\label{AA52}  Let $G=S_{n}$ 
and $H=C(n,k;r)$ for $2\leq r<k<n$. Then the eigenvalues of $H^{+}$ on the natural module
$M$ are \\[5pt]
$\mu_1=|H|=(k-1)!{n-r\choose k-r}$,\,\,\,\,
$\mu_{2}=(k-2)! {n-r \choose k-r} \frac{1}{n-r} \big((k-1)(n-k) - \frac{(k-r-1)(k-r)}{n-r-1}\big)$\\[5pt]
$\mu_3=(k-2)!{n-r\choose k-r}\big(\frac{r(n-k)}{n-r}-1\big)$ \,and\, $\mu_4=-(k-2)!{n-r\choose k-r}.$
\end{theo}

In the remainder of this section let $k,$ $G$ and $H$ be as in the theorem and put  $\Gamma=\Cay(G,H).$ We exhibit two equitable partitions of $\gG$ so that  their associated  eigenvalues  belong to \\$\{\mu_{1},\, \mu_{2},\,\mu_{3},\,\mu_{4}\}.$ The proof is completed at the end of the section.

First let $1\leq r$ and let $K_{1}=K=\Sym(1\ld n-1)$ be the stabilizer in $G$ of $n.$ Define  the partition $P_{1}$ of the vertices of $\gG$ by
\begin{eqnarray} V_{1}&=&K=\{g\in G:\,\,g(n)=n\}; \nonumber\\ V_{2}&=& (n,1)K\cup\dots\cup (n,r)K=\{g\in G:\,\,g(n)\in\{1,...,r\}\}; \nonumber\\
V_{3}&=& (n,r+1)K\cup\dots\cup (n,n-1)K=\{g\in G:\,\,g(n)\in\{r+1,...,n-1\}\}.\end{eqnarray}

Below ${a\choose b}$ for integers $a\geq b>0$ denotes the number of choices of $b$ elements from a set of $a$ elements. For uniformity we need to make the cases with $b=0$ and $b=-1$ meaningful: we define
${a\choose 0}$ to be 1, and ${a\choose -1}$ to be $0$.

\bl{ep1} Let $1\leq r<k<n.$ Then $P_{1}$ is an equitable partition of $\gG$ with matrix

\vspace{-0.5cm}
$$B_1=\begin{pmatrix}     \,\,\,(k-1)!{n-r-1\choose k-r}  &   r(k-2)!{n-r-1\choose k-r-1}        &  (n-r-1)(k-2)!{n-r-2\choose k-r-2}  \,\,\,\cr   &&\cr
(k-2)!{n-r-1\choose k-r-1}    &     (r-1)(k-2)!{n-r\choose k-r}  &     (n-r-1)(k-2)! {n-r-1\choose k-r-1}      \cr   &&\cr
(k-2)!{n-r-2\choose k-r-2}   &  r(k-2)!{n-r-1\choose k-r-1}           & (n-r-2)(k-2)!{n-r-2\choose k-r-2}+(k-1)!{n-r-1\choose k-r}     \end{pmatrix}$$

$$=(k-2)!\begin{pmatrix}     \,\,\,(k-1){n-r-1\choose k-r}  &   r{n-r-1\choose k-r-1}        &  (k-r-1){n-r-1\choose k-r-1}  \,\,\,\cr   &&\cr
{n-r-1\choose k-r-1}    &     (r-1){n-r\choose k-r}  &     (n-r-1) {n-r-1\choose k-r-1}      \cr   &&\cr
{n-r-2\choose k-r-2}   &  r{n-r-1\choose k-r-1}           & (n-r-2){n-r-2\choose k-r-2}+(k-1){n-r-1\choose k-r}     \end{pmatrix}.$$

The \eis of $B_1$ are  $\mu_1=(k-1)!{n-r\choose k-r},$\\
$\mu_{2}=(k-2)! {n-r \choose k-r} \frac{1}{n-r} \big((k-1)(n-k) - \frac{(k-r-1)(k-r)}{n-r-1}\big)$ and \\
$\mu_3=(k-2)!{n-r\choose k-r}(r-1-r\frac{k-r}{n-r}\,)=
(k-2)!{n-r\choose k-r}\big(\frac{r(n-k)}{n-r}-1\big).$
\el

\bp Put $F=\Sym(1\ld r)\times \Sym(r+1\ld n-1)$ and let $F\times K$ act on $V$ as in Lemma~\ref{AA24}. Then $F\times K$ stabilizes each $V_{i}$ and acts transitively on it.  Hence $P_{1}$ is equitable by Lemma~\ref{AA23}.

Let $B_1=(\be_{ij})$ for $1\leq i,j\leq 3$ and so $\be_{ij}=  |N(v_{i})\cap V_{j}|=|Hv_{i}\cap V_{j}|=|H\cap V_{j}v_{i}^{-1}|$ for $v_{i}\in V_{i}.$ Taking $v_{1}=1_{G},\,v_{2}=(n,1)$ and $v_{3}=(n,n-1)$ we have

($i=1$):\,\, $\be_{11}=|H\cap K|=(k-1)!{n-r-1\choose k-r};$\\
\smallskip
 $\be_{12}=|H\cap V_{2}|= r|H\cap Kv_{2}|=r(k-2)!{n-r-1\choose k-r-1};$\\
\smallskip
 $\be_{13}=|H\cap V_{3}|=(n-r-1)\,|H\cap v_{3}K|=(n-r-1)(k-2)!{n-r-2\choose k-r-2}$.

($i=2$):\,\,$\be_{21}=|H\cap K(n,1)|=|\{h\in H\,:\,h(1)=n \}|=(k-2)!{n-r-1\choose k-r-1};$   \\
\smallskip $\be_{22}=|H\cap V_{2}v_{2}|=|H\cap (n,1)Kv_{2}|+|H\cap (n,2)Kv_{2}|+\dots +|H\cap (n,r)Kv_{2}|=\\ {}\hskip1cm=0+(r-1)\,|(n,2)Kv_{2}\cap H|=(r-1)(k-2)!{n-r\choose k-r};$ \\
\smallskip$\be_{23}=|H\cap V_{3}v_{2}|=|H\cap (n,r+1)K(n,1)|+\dots+|H\cap (n,n-1)K(n,1)|=
(n-r-1)(k-2!) {n-r-1\choose k-r-1}$.

$(i=3)$:\,\,  $\be_{31}=|H\cap V_{1}v_{3}|=(k-2)!{n-r-2\choose k-r-2};$\\
\smallskip
$\be_{32}=|H\cap V_{2}v_{3}|=r(k-2)!{n-r-1\choose k-r-1};$ \\
\smallskip
$\be_{33}=|H\cap V_{3}v_{3}|=(n-r-2)(k-2)!{n-r-2\choose k-r-2}+(k-1)!{n-r-1\choose k-r}$ .

This gives the matrix above after some further transformations. We have confirmed the eigenvalues using a symbolic calculator such as Wolfram$\!|\!$Alpha~\cite{WA}.   \enp

Next let $2\leq r$ and let $K_{2}=K=\Sym(2\ld n)$ be the stabilizer in $G$ of $1.$ Define  the partition $P_{2}$ of the vertices of $\gG$ by  \begin{eqnarray}
V_{1}&=&K=\{g\in G:\,\,g(1)=1\} \nonumber\\
V_{2}&=& (2,1)K\cup\dots\cup (r,1)K=\{g\in G:\,\,g(1)\in\{2,...,r\}\} \nonumber\\
V_{3}&=& (r+1,1)K\cup\dots\cup (n,1)K=\{g\in G:\,\,g(1)\in\{r+1,...,n\}\}\end{eqnarray}

\bl{ep2} Let $2\leq r<k<n.$ Then $P_{2}$ is an equitable partition of $\gG$ with matrix

\vspace{-0.3cm}
$$B_2=\begin{pmatrix}     \,\,\, 0  & (r-1)(k-2)!{n-r\choose k-r}  &  (k-r)(k-2)!{n-r\choose k-r} \,\,\,\cr   &&\cr

(k-2)!{n-r\choose k-r}   &   (k-2)!  (r-2){n-r\choose k-r} &    (k-r)(k-2)!{n-r\choose k-r}     \cr   &&\cr

(k-2)!  {n-r-1\choose k-r-1}  &  (r-1)(k-2)!{n-r\choose k-r-1}           &(k-2)!(n-r-1)  {n-r-2\choose k-r-2}   +(k-1)!{n-r-1\choose k-r}    \end{pmatrix}$$
$$=(k-2)!{n-r\choose k-r} \begin{pmatrix}     \,\,\, 0  & r-1  &  k-r \,\,\,\cr   &&\cr
1   &     r-2 &    k-r     \cr   &&\cr
\,\,\frac{k-r}{n-r} & \,\, (r-1)\frac{k-r}{n-r}          & \,\,\,k-1 -r\frac{(k-r)}{n-r}  \,\,
\end{pmatrix}. $$

The \eis of $B_2$ are $\mu_1=(k-1)!{n-r\choose k-r}$,\,  $\mu_3=(k-2)!{n-r\choose k-r}\big(\frac{r(n-k)}{n-r}-1\big)$ and \\$\mu_4=-(k-2)!{n-r\choose k-r}$.
\el

\bp Put $F=\Sym(2\ld r)\times \Sym(r+1\ld n)$ and let $F\times K$ act on $V$ as in Lemma~\ref{AA24}. Then $F\times K$ stabilizes each $V_{i}$ and acts transitively on it.  Hence $P_{1}$ is equitable by Lemma~\ref{AA23}.

Let $B_2=(\be_{ij})$ for $1\leq i,j\leq 3$ and so $\be_{ij}=  |N(v_{i})\cap V_{j}|=|Hv_{i}\cap V_{j}|=|H\cap V_{j}v_{i}^{-1}|$ for $v_{i}\in V_{i}.$ Taking $v_{1}=1_{G},\,v_{2}=(2,1)$ and $v_{3}=(n,1)$ we have

\smallskip
($i=1$):\,\, $\be_{11}=|H\cap K|=0;$ \\
\smallskip
 $\be_{12}=|H\cap V_{2}|=(r-1)(k-2)!{n-r\choose k-r};$\\
\smallskip
$\be_{13}=|H\cap V_{3}|=(n-r)(k-2)!{n-r-1\choose k-r-1}=(k-r)(k-2)!{n-r\choose k-r}$. 

\smallskip
($i=2$):\,\,$\be_{21}=|H\cap K(12)|=(k-2)!{n-r\choose k-r};$  \\
\smallskip $\be_{22}=|H\cap V_{2}(1,2)|=(r-2)(k-2)!{n-r\choose k-r};$\\
\smallskip
$\be_{23}=|H\cap V_{3}(1,2)|=(n-r)(k-2)!{n-r-1\choose k-r-1}=(k-r)(k-2)!{n-r\choose k-r}.$

\smallskip
($i=3$):\,\,$\be_{3,1}=|H\cap K(1,n)| =(k-2)!  {n-r-1\choose k-r-1} ;$\\
\smallskip$\be_{32}=|H\cap V_{2}(1,n)|=   (k-2)!  (r-1){n-r-1\choose k-r-1}   ;$\\
\smallskip$\be_{33}=|H\cap V_{3}(1,n)|=  (k-2)!(n-r-1)  {n-r-2\choose k-r-2}   +(k-1)!{n-r-1\choose k-r} =  \\{} \hskip1cm =(k-2)!(k-r-1)  {n-r-1\choose k-r-1} +(k-1)!{n-r-1\choose k-r}   .$

This gives the matrix above after some further transformations. We have confirmed the eigenvalues using a symbolic calculator.   \enp

We emphasize that $P_{1}$ and $P_{2}$ are equitable partitions of $V=G$ and that both are refined by the coset partition $P_{0}=g_{1}K\cup g_{2}K\cup\dots\cup g_{n}K $  which gives rise to the natural module $M.$ Therefore by Lemma~\ref{AA22} all eigenvalues $\mu_{i}$ for $1\leq i\leq 4$ in Theorem~\ref{AA52} are eigenvalues of $\Cay(G,H)$ for $H=C(n,k;r).$

\smallskip
{\it Proof of Theorem {\rm \ref{AA52}}.}~Let $1\leq r<k<n,$ $H=C(n,k,r)\subset Sym(1\ld n)$ and let $\mu_{1},\,\mu_{2},\,\mu_{3},\,\mu_{4}$ be the eigenvalues stated in Lemmas~\ref{ep1} and~\ref{ep2}.  By Lemma~\ref{AA22} we have $\mu_{1}=\gl_{1}(\gG).$ Since $G/K_{1}$ and $G/K_{2}$ are permutationally isomorphic, by Lemma 2.2, $\mu_{1},\,\mu_{2},\,\mu_{3},\,\mu_{4}$ are distinct eigenvalue of $H^{+}$ on $M.$  Next we note that  $H$ is invariant under conjugation by  $F=\Sym(1\ld r)\times \Sym(r+1\ld n).$ By Lemma~\ref{ne1}, $H^{+}$ has at most four eigenvalues and hence the result. \dne

{\it Proof of Theorem~{\rm {\ref{th13-}}.}}~By the comment above the theorem follows from Theorem~{\ref{AA52} and Lemma~\ref{AA231}. \dne


\section{Proof of Theorem \ref{th13}}

In this section let $4<n,$ $2\leq r\leq n-2$ and $H=C(n,r+1;r)\subset S_n$. Let $M$ be the natural $S_n$-module.

\bl{me4} 
The  multiplicities of the \eis of $H^+$ on M are as follows:
{\small
\begin{center}
\begin{tabular}{|c|c|c|c|c|}
\hline
{\rm Eigenvalue of $H^+$ on} $M$&$r!(n-r)$&$r!(n-r-1)$&$(r-1)(r-1)!(n-r-1)-r!$&$-(r-1)!(n-r)$\cr
\hline
{\rm the multiplicity}&$1$&$n-r-1$&$1$&$r-1$\cr
\hline
\end{tabular}
\end{center}}\el

\begin{proof} The \eis of $|H^+|$ on $M$ are determined in  Theorem~\ref{AA52}, where we take $k=r+1$. 
Note that the trace of $h\in H$ on $M$ equals $n-r-1$. So the trace of $H^+$ equals $|H^+|(n-r-1)=r!(n-r)(n-r-1)$. We know that the \ei $|H^+|$ has \mult 1 (Theorem \ref{wk1}), and the multiplicities $x,y,z$ of the three other \eis 
are in the set  $\{1,n-r-1,r-1\}$ by Lemma \ref{ne1}. So we have
{\small
$$|H^+|(n-r-1)=r!(n-r)(n-r-1)=r!(n-r)+xr!(n-r-1)+y((r-1)(r-1)!(n-r-1)-r!)-z(r-1)!(n-r).$$}\\[-10pt]
If we take $x=n-r-1$, $y=1$, $z=r-1$ then the equality holds. In addition, the equality fails for any other choice of $\{x,y,z\}\in\{1,n-r-1,r-1\}$. Whence the result. \end{proof}

\begin{corol}\label{co2} Let $M'$ be an $S_n$-module
whose composition factor dimensions are at most $n-1$. Then the second largest \ei of $H^+$ on $M'$ does not exceed $r!(n-r-1)$.\end{corol}

\begin{proof}  Let $\rho$ be an \ir characters of $S_n$ of degree at most   $n-1$. Then $\rho $ is well known to be a constituent either of $M$ or $M\otimes \si$, where $\si$ denotes the sign \rep of $S_n$.
 \itf the \eis of $H^+$ on $M'$ are in the list of Lemma \ref{ne1} if $r$ is even, and also their negatives, if $r$ is odd. So the result follows as $r!(n-r-1)>(r-1)!(n-r)$.
\end{proof}

\bl{ss4} Let $G=S_n$ and let L be an \ir $ G$-module. Then there is a basis of L such that
the matrices $g+g\up$ $(g\in G)$ are symmetric. Consequently, if $H\subset G$ is a subset such that $H=H\up$ then the matrix of $H^+$ on L is symmetric.\el
\begin{proof} It is well known that there is a basis of $L$ such that
all matrices $g\in G$ are orthogonal, that is, if $g_L$ is the matrix of $g$ on $L$ then
$g_L^T=(g_L)\up=(g\up)_L$, where $(g_L)^T$ is the transpose of $g_L$. Then $g_L+g_L\up=g_L+(g_L)^T$, whence the claim.\end{proof}

\bl{01a} Let L be an \ir $S_n$-module and $\lam$ an \ei of $H^+$ on L. Then $\lam\leq |H|$
and $\lam=|H|$ implies $\dim L=1$.\el
\bp If $H$ is connecting then the lemma follows from Theorem \ref{wk1}. Indeed,
$\lam=|H|$ if $L$ is the trivial module, and the trivial module has \mult 1 in $\cp S_n$.
If $H$ is not a connecting set then $H$ is a connecting set of $A_n$ (Lemma \ref{n0}), and the same argument
is valid if $L$ is an $A_n$-module. So $|H|$ is an \ei of $H^+$ on $L$ \ii the restriction of
$L$ to $A_n$ contains the trivial $A_n$-module. This is well known to imply that $\dim L=1$.  \enp

 Set $H_{n-1}=H\cap K$, where $K=\Sym(1\ld n-1)$ 
 and $E=H\setminus H_{n-1}\subset \Sym(1\ld n-2,n)$. Then $|H_{n-1}|=r!(n-r-1)$ and $|E|=|H| -|H_{n-1}|=r!$

\begin{theo}\label{th3} Let L be an \ir $S_n$-module and $\dim L>n-1$. Let $\lam$
be an \ei of $H^+$ on L. Then $\lam\leq r!(n-r-1)$.
\end{theo}

\bl{ii2}  Theorem {\rm \ref{th3}} is true for $r=n-2>2$.\el

\bp The case of $n=5$ can be checked straightforwardly. Let $n>5$.

Note that $r!(n-r-1)$ equals $(n-2)!$ for $r=n-2$.
We have   $H_{n-1}=\{(n-1,i_1\ld i_{n-2})\}$ and $E=\{(n,i_1\ld i_{n-2})\}$, where $\{i_1\ld i_{n-2}\}=\{1\ld n-2\}$. So $H_{n-1}$ is a conjugacy class in $Sym\{1\ld n-1\}$. 
Let $\al_1,\be_1$ be the maximum of \eis of $H_{n-1}^+,E^+$, respectively, on $L$. By the branching rule,   as an $S_{n-1}$-module $L$ has no \ir constituent of dimension 1. So, by Lemma \ref{jn1}, if $n-1>4$ is even then $\al_1\leq (n-3)!$,  and if $n-1$ is odd then this is at most $\al_1\leq 2(n-4)!$. As $E$ is a conjugacy class of $Sym(1\ld n-2,n)\cong S_{n-1}$, the same conclusion holds for the \ei $\be_1$ of $E^+$and $\al_1=\be_1$. By Lemma \ref{ss4}, we can assume that the matrices of $H^+, $  $H_{n-1}^+$, $E$ on $L$ are symmetric. If $\gamma_1$  is the largest \ei of $H^+ $ on $L$, then, by the Weyl inequality,  $\gamma_1\leq \al_1+\be_1= 2\al_1$.
  So $\gamma_1\leq 4(n-4)!$  if $n-1$ is odd, and $\gamma_1\leq 2(n-3)!$ if $n-1$ is even. Obviously, $\gamma_1<(n-2)!$.\enp

\bp[Proof of Theorem {\rm \ref{th3}}] As mentioned above,  as an $S_{n-1}$-module $L$ has no composition factor of dimension $1.$ Therefore, by Lemma \ref{01a}, $H_{n-1}^+$ does not have \ei $|H_{n-1}^+|=r!(n-r-1)$ on $L$.
We use induction on $n-r$ with $n-r=2$ as base of induction. The case with $r=n-2$  is settled in Lemma \ref{ii2}. By the induction assumption, if $\mu$ is an \ei of $H_{n-1}^+$ on $L$ then $\mu\leq r!(n-r-2)$.

  Let $\nu$ be the largest \ei of $E^+$ on $L$. Then $\nu\leq |E|=r!$ (by Theorem \ref{wk1} applied to the graph $(E,V)$.)
By the Weyl inequality (\ref{eH2}) applied to $L,$ we have $\lam\leq \mu+\nu\leq r!(n-r-2)+r!=r!(n-r-1)$, as required. \enp

\bp[Proof of Theorem {\rm \ref{th13}}] If $G=S_n$ then then this is equivalent to Theorem \ref{th3}.
For $A_n$ the result follows from the statement for $S_n$ and  Lemma \ref{an5}. \enp


\begin{thebibliography}{1245}

\bibitem{An} N. Alon, Eigenvalues and Expanders,  {\it Combinatorica} 6(1986), $83-96$.

\bibitem{We} W.  Barrett, Hermitian and Positive Definite Matrices, Chapter 8
in: {\it Handbook in Linear Algebra}, Chapman and Hall, London, 2007.

\bibitem{BH} A.E. Brouwer and W.H. Haemers, {\it Spectra of Graphs},  Springer, Berlin, 2012. 

\bibitem{Dem} P. Dembowski, {\it Finite Geometries}, Springer Verlag, Berlin, 1968.

\bibitem{HLW} S. Hoory, N. Linial, A. Wigderson, Expander Graphs and their Applications,  {\it Bull. Amer. Math. Soc.} 43(2006), $439-561$.

\bibitem{HH} X. Huang and Q. Huang, The second largest \ei of some Cayley graphs on alternating groups,
{\it J. Algebraic Combinatorics} 50(2019), $99-111$.

\bibitem{J} G. James, {\it The Representation Theory of the Symmetric Groups}, Springer, Berlin, 1978.

\bibitem{JK} G. James and A.   Kerber, {\it  The Representation  Theory of the symmetric group}, Addison-Wesley, London, 1981.

\bibitem{KY} K. Kalpakis  and Y. Yesha, On the bisection width of the transposition network, {\it Networks} 29(1997), $69-76$.


\bibitem{kr} M. Krebs and A. Shaheen, {\it Expander Families and Cayley Graphs}, Oxford Univ. Press, Oxford, 2011.

\bibitem{LZ} X. Liu and S. Zhou, Eigenvalues of Cayley graphs, ArXiv 1809.09829v2 [math.CO], 22 Jan 2019.

\bibitem{Lu} A. Lubotzky, Cayley Graphs: Eigenvalues, Expander and Random Walks;
in:  {\it Surveys in Combinatorics BCC} 1995, ed. P. Rowlinson, LMS Mathematical Lecture Notes, Vol. 218 (1995), $155-191$.


\bibitem{Sta} Z. Stanic,  {\it Inequalities for Graph Eigenvalues}, LMS Lecture Note Series, Vol 423, Cambridge Univ. Press, Cambridge, 2015.


\bibitem{SZ} J. Siemons and A.E. Zalesski, Remarks on singular Cayley graphs and vanishing elements of simple groups, {\it J. Algebraic Combinatorics} 50(2019), $379 - 401$.

\bibitem{Weyl} H.Weyl, Das asymptotische Verteilungsgesetz der Eigenwerte linearer partieller Differentialgleichungen,  {\it Math. Ann.} 71(1912), $441- 479$.


\bibitem{WA} Wolfram$\!|\!$Alpha,\tt{https://www.wolframalpha.com/widgets/gallery/view.jsp?\\id=9aa01caf50c9307e9dabe159c9068c41}
\end{thebibliography}
\end{document}